\documentclass{article}
\usepackage{amsmath}
\usepackage{amscd}
\usepackage{amsthm}
\usepackage{verbatim}
\usepackage{amssymb}
\usepackage{amsfonts}
\usepackage{mathrsfs}
\usepackage{lscape}
\usepackage{stmaryrd}
\usepackage[usenames]{color}
\usepackage[all,color]{xy}

\usepackage{graphicx}
\newtheorem{thm}{Theorem}[section]

\newtheorem{lem}[thm]{Lemma}
\newtheorem{prop}[thm]{Proposition}

\newtheorem{defn}[thm]{Definition}

\numberwithin{equation}{section}

\newcommand{\Hom}{\operatorname{Hom}}
\newcommand{\cone}{\operatorname{cone}}

\newcommand{\K}{\operatorname{K}}

\newcommand{\id}{\operatorname{id}}

\title{A Note on the Grothendieck Group of an Additive Category}
\author{David E. V. Rose \\
{ \sl \small Mathematics Department,}\\
      { \sl \small Duke University, Durham, NC 27708-0320, USA}
         \\
      {\tt \small email: derose@math.duke.edu} \\}
\date{ }

\begin{document}

\maketitle

\begin{abstract}
There are two abelian groups which can naturally be associated to an additive category $\mathcal{A}$:
the split Grothendieck group of $\mathcal{A}$ and the triangulated Grothendieck group of the homotopy category 
of (bounded) complexes in $\mathcal{A}$. We prove that these groups are isomorphic. 
Along the way, we deduce that the `Euler characteristic' of a complex in $\mathcal{A}$ 
is invariant under homotopy equivalence.
\end{abstract}

\section{Introduction}
A categorification of an algebraic structure is typically given by an additive category (often possessing additional 
structure) from which the original structure can be recovered by taking the Grothendieck group; see for 
instance \cite{KhovanovMS} for the abelian case.
In certain categorifications of quantum knot invariants, the categorification is accomplished by first finding an 
additive category which categorifies an algebraic structure and then 
passing to the homotopy category of complexes to give the categorification of the knot invariant 
(see \cite{BarNatan2} and \cite{MorrisonNieh}).
The categorified knot invariant decategorifies to give the original knot invariant by taking the 
`Euler characteristic' of the complex, the alternating sum of terms of the complex, 
viewed as an element of the split Grothendieck group of the additive category. 
Since the homotopy category is triangulated, 
the natural decategorification of this category is its triangulated Grothendieck group. This posits the question, 
are these two Grothendieck groups isomorphic? This question can equivalently be stated: is the Euler characteristic 
of a complex invariant under homotopy equivalence?

We answer both these questions in the affirmative:

\begin{thm}\label{1}
Let $\mathcal{A}$ be an additive category and $K^b(\mathcal{A})$ denote the homotopy category of bounded complexes 
in $\mathcal{A}$. 
The split Grothendieck group of $\mathcal{A}$ is isormophic to the 
triangulated Grothendieck group of $K^b(\mathcal{A})$.
\end{thm}

\begin{thm}\label{2}
Let $A^{\cdot} \simeq B^{\cdot}$ be homotopy equivalent complexes in $K^b(\mathcal{A})$, then
\[
\sum_{i= -\infty}^{\infty} (-1)^i \langle A^i \rangle 
= \sum_{i= -\infty}^{\infty} (-1)^i \langle B^i \rangle,
\]
where $\langle - \rangle$ denotes the corresponding 
element in the split Grothendieck group of $\mathcal{A}$.
\end{thm}

We present the relevant background on additive categories and Grothendieck groups in Section \ref{background}. 
In Section \ref{lastsec} we prove Theorems \ref{1} and \ref{2} and discuss a slight generalization of 
Theorem \ref{2} which is used in \cite{Rose2}. 

\noindent \textbf{Acknowledgments:} I would like to thank Ezra Miller for a helpful conversation and 
Scott Morrison for useful correspondence. I would also like to thank my advisor Lenny Ng for his 
continued guidance. 
The author was partially supported by NSF grant DMS-0846346 during the completion of this work.

\section{Background}\label{background}
Let $\mathcal{A}$ be an additive category. Recall that this means that $\mathcal{A}$ has a zero object, 
finite biproducts, 
and that $\Hom_{\mathcal{A}}(A_1,A_2)$ is an abelian group for any objects $A_1,A_2$ in $\mathcal{A}$ with addition 
distributing over composition.
\begin{defn}\label{defn1}
The \emph{split Grothendieck group} of $\mathcal{A}$, denoted $\K_{\oplus}(\mathcal{A})$, is the abelian group 
generated by isomorphism classes $\langle A \rangle$ of objects in $\mathcal{A}$ modulo the relations 
$\langle A_1\oplus A_2 \rangle = \langle A_1 \rangle + \langle A_2 \rangle$ 
for all objects $A_1,A_2$ in $\mathcal{A}$.
\end{defn}
Recall that the Grothendieck group of an abelian category is the abelian group generated by isomorphism classes 
$\langle A \rangle$ of objects modulo the relations 
$\langle A_2 \rangle = \langle A_1 \rangle + \langle A_3 \rangle$ for every short exact sequence 
\[
0 \to A_1 \to A_2 \to A_3 \to 0
\] 
in $\mathcal{A}$.
We can think of Definition \ref{defn1} as the analog of this notion in an additive category where we impose relations 
corresponding to the only notion of exact sequence that makes sense, the split exact sequences
\[
0 \to A_1 \to A_1\oplus A_2 \to A_2 \to 0.
\]

Suppose now that $\mathcal{C}$ is not only additive, but triangulated.
\begin{defn}\label{defn2}
The \emph{triangulated Grothendieck group}, denoted $\K_{\triangle}(\mathcal{C})$, is the abelian group 
generated by isomorphism classes $\langle C \rangle$ of objects in $\mathcal{C}$ quotiented by the relation 
$\langle C_2 \rangle = \langle C_1 \rangle + \langle C_3 \rangle$ 
for all distinguished triangles 
$C_1\to C_2\to C_3$.
\end{defn} 
Again, we think of distinguished triangles as the analogs of short exact sequences in $\mathcal{C}$.

\section{Grothendieck Groups of Additive Categories}\label{lastsec}
Now fix an additive category $\mathcal{A}$. Let $K^b(\mathcal{A})$ denote the homotopy category of bounded 
complexes in $\mathcal{A}$ (we apologize for the confusing, but somewhat standard, notation). 

Let $A^{\cdot} = \left(A^k \stackrel{d^k}{\longrightarrow} \cdots \stackrel{d^{l-1}}{\longrightarrow} A^l\right)$ 
be a bounded complex and 
let $A[m]^{\cdot}$ denote the complex shifted up by $m$ in homological degree. We will underline 
the term in homological degree zero when it is not clear from the context.
The distinguished triangle 
\[
A^{\cdot} \to 0 \to A[-1]^{\cdot}
\]
gives that 
\begin{equation}\label{rel1}
\langle A[-1]^{\cdot} \rangle = -\langle A^{\cdot} \rangle
\end{equation}
and the triangle
\[
A^k \to
\left(\underline{A}^{k+1} \stackrel{d^{k+1}}{\longrightarrow} \cdots \stackrel{d^{l-1}}{\longrightarrow} A^l\right)\to
A[-k-1]^{\cdot}
\]
shows (via induction) that 
\begin{equation}\label{rel2}
\langle A^{\cdot} \rangle = \chi(A^{\cdot})  
\end{equation}
in $\K_{\triangle}(K^b(\mathcal{A}))$.
Here $\chi(A^{\cdot}) := \sum_{i=-\infty}^{\infty} (-1)^i \langle A^i \rangle$ and 
$A^i$ is shorthand for the complex with the object $A^i$ in degree zero and all other terms zero.
From this we see that $\K_{\triangle}(K^b(\mathcal{A}))$ and $\K_{\oplus}(\mathcal{A})$ are generated 
by the same elements. 

Given complexes $A_1^{\cdot}$ and $A_2^{\cdot}$, the distinguished triangle
\[
A_1^{\cdot} \to (A_1 \oplus A_2)^{\cdot}\to A_2^{\cdot}
\]
shows that 
\begin{equation}\label{rel3}
\langle (A_1\oplus A_2)^{\cdot}\rangle 
= \langle A_1^{\cdot} \rangle + \langle A_2^{\cdot} \rangle.
\end{equation}
It follows that there is a surjective map 
$\K_{\oplus}(\mathcal{A}) \to \K_{\triangle}(K^b(\mathcal{A}))$.

To prove Theorem \ref{1}, it suffices to show that this map is injective or equivalently that there 
are no additional relations imposed on $\K_{\triangle}(K^b(\mathcal{A}))$ other than those given in equations 
\eqref{rel1}, \eqref{rel2}, and \eqref{rel3}. 
Given a map $A_1 \stackrel{f}{\rightarrow} A_2$, these equations show that 
\begin{align}
\langle \cone(f)^{\cdot} \rangle &= 
\displaystyle \sum_{j=\infty}^{\infty} \Big( (-1)^j \langle A_2^j \rangle 
+ (-1)^{j+1} \langle A_1^j \rangle \Big) \label{eq1} \\
&= \langle A_2^{\cdot} \rangle - \langle A_1^{\cdot} \rangle \notag
\end{align}
so distinguished triangles of the form 
\begin{equation}\label{st}
A_1^{\cdot} \stackrel{f}{\rightarrow} A_2^{\cdot} \to \cone(f)^{\cdot}
\end{equation}
contribute no new relations. 
Since all distinguished triangles are isomorphic to those of the form \eqref{st} and isomorphism in $K^b(\mathcal{A})$ 
is homotopy equivalence, it suffices to prove Theorem \ref{2}.

To this end, suppose that $\varphi:A_1^{\cdot} \to A_2^{\cdot}$ is a homotopy equivalence. The following result 
from \cite{Spanier} is given in the setting of the category of abelian groups, 
but the proof sketched there carries over to arbitrary 
additive categories. We provide the details of the proof for completeness.
\begin{lem}\label{lemma1}
A chain map $\varphi:A_1^{\cdot} \to A_2^{\cdot}$ is a homotopy equivalence iff $\cone(\varphi)^{\cdot}$ is 
null-homotopic.
\end{lem}

\begin{proof}
Let $\varphi:A_1^{\cdot}\to A_2^{\cdot}$ be a homotopy equivalence, so there exists a chain map 
$\psi:A_2^{\cdot}\to A_1^{\cdot}$ so that
\[
\varphi^j \psi^j - id_2^j = d_2^{j-1} H_2^j + H_2^{j+1} d_2^j
\]
and
\[
\psi^j \varphi^j - id_1^j = d_1^{j-1} H_1^j + H_1^{j+1} d_1^j
\]
for maps $H_1^j: A_1^j \to A_1^{j-1}$ and $H_2^j: A_2^j \to A_2^{j-1}$. We now construct maps 
$H^j:\cone(\varphi)^j \to \cone(\varphi)^{j-1}$ so that $id_{\cone(\varphi)}^j = d^{j-1}H^j + H^{j+1}d^j$ where 
\[
d^j = \begin{pmatrix} -d_1^{j+1} & 0 \\
-\varphi^{j+1} & d_2^j \end{pmatrix}.
\] 
Let 
\[
H^j = \begin{pmatrix} H_1^{j+1} + \psi^j H_2^{j+1}\varphi^{j+1} - \psi^j \varphi^j H_1^{j+1} & -\psi^j \\
H_2^j H_2^{j+1} \varphi^{j+1} - H_2^j \varphi^j H_1^{j+1} & -H_2^j \end{pmatrix}
\]
and denote $M^j = d^{j-1}H^j - H^{j+1} d^j$. We now compute the entries $M_{(kl)}^j$ of this matrix:
\begin{align*}
M_{(11)}^j &= -d^j H_1^{j+1} - d_1^j \psi^j H_2^{j+1} \varphi^{j+1} + d_1^j \psi^j \varphi^j H_1^{j+1} - H_1^{j+2}d_1^{j+1} \\
         &  \ \ \ \ \ - \psi^{j+1} H_2^{j+2} \varphi^{j+2} d_1^{j+1} 
         + \psi^{j+1} \varphi^{j+1} H_1^{j+2} d_1^{j+1} + \psi^{j+1} \varphi^{j+1}    \\ 
         &= id_1^{j+1} - \psi^{j+1} \varphi^{j+1} - d_1^j \psi^j H_2^{j+1} \varphi^{j+1} + d_1^j \psi^j \varphi^j H_1^{j+1}  \\
         &  \ \ \ \ \ - \psi^{j+1} H_2^{j+2} \varphi^{j+2} d_1^{j+1} 
         + \psi^{j+1} \varphi^{j+1} H_1^{j+2} d_1^{j+1} + \psi^{j+1} \varphi^{j+1}    \\
         &= id_1^{j+1} - \psi^{j+1}(d_2^j H_2^{j+1} + H_2^{j+2} d_2^{j+1})\varphi^{j+1} \\
         &  \ \ \ \ \ + \psi^{j+1} \varphi^{j+1}(d_1^j H_1^{j+1} + H_1^{j+2} d_1^{j+1}) \\
         &= id_1^{j+1} - \psi^{j+1}(\varphi^{j+1} \psi^{j+1} - id_2^{j+1})\varphi^{j+1} \\
         &  \ \ \ \ \ + \psi^{j+1} \varphi^{j+1}(\psi^{j+1} \varphi^{j+1} - id_1^{j+1}) \\
         &= id_1^{j+1},
\end{align*}

\[
M_{(12)}^j = d_1^j\psi^j - \psi^{j+1}d_2^j = 0,
\]

\begin{align*}
M_{(21)}^j &= -\varphi^j H_1^{j+1} - \varphi^j \psi^j H_2^{j+1}\varphi^{j+1} + \varphi^j \psi^j \varphi^j H_1^{j+1} 
           + d_2^{j-1} H_2^j H_2^{j+1} \varphi^{j+1} \\
         &= -d_2^{j-1} H_2^j \varphi^j H_1^{j+1} - H_2^{j+1} H_2^{j+2} \varphi^{j+2} d_1^{j+1} \\
         & \ \ \ \ \ + H_2^{j+1} \varphi^{j+1} H_1^{j+2} d_1^{j+1} + H_2^{j+1} \varphi^{j+1} \\
         &= (id_2^j + d_2^{j-1}H_2^j - \varphi^j \psi^j)H_2^{j+1}\varphi^{j+1} 
           + (\varphi^j \psi^j - d_2^{j-1}H_2^j - id_2^j)\varphi^j H_1^{j+1} \\
         & \ \ \ \ \ - H_2^{j+1} H_2^{j+2} \varphi^{j+2} d_1^{j+1} + H_2^{j+1}\varphi^{j+1} H_1^{j+2} d_1^{j+1} \\
         &= -H_2^{j+1} d_2^j H_2^{j+1}\varphi^{j+1} + H_2^{j+1} d_2^j \varphi^j H_1^{j+1} \\
         & \ \ \ \ \ - H_2^{j+1} H_2^{j+2} \varphi^{j+2} d_1^{j+1} + H_2^{j+1}\varphi^{j+1} H_1^{j+2} d_1^{j+1} \\ 
         &= -H_2^{j+1}(d_2^j H_2^{j+1}+H_2^{j+2} d_2^{j+1})\varphi^{j+1} \\
         & \ \ \ \ \ + H_2^{j+1}\varphi^{j+1}(d_1^jH_1^{j+1}+H_1^{j+2}d_1^{j+1}) \\
         &= -H_2^{j+1}(\varphi^{j+1}\psi^{j+1}-id_2^{j+1})\varphi^{j+1} + H_2^{j+1}\varphi^{j+1}(\psi^{j+1}\varphi^{j+1}-id_1^j) \\
         &= 0,
\end{align*}

\[
M_{(22)}^j = \varphi^j\psi^j - d_2^{j-1}H_2^j - H_2^{j+1}d_2^j = id_2^j.
\]
This shows that $\cone(\varphi)^{\cdot} \simeq 0$.

Conversely, let $\cone(\varphi)^{\cdot}$ be null-homotopic then
\[
id_{\cone(\varphi)}^j = \begin{pmatrix} -d_1^j & 0 \\ -\varphi^j & d_2^{j-1} \end{pmatrix}
\begin{pmatrix} h_{11}^{j+1} & h_{12}^j \\ h_{21}^{j+1} & h_{22}^j \end{pmatrix} 
+ \begin{pmatrix} h_{11}^{j+2} & h_{12}^{j+1} \\ h_{21}^{j+2} & h_{22}^{j+1} \end{pmatrix}
\begin{pmatrix} -d_1^{j+1} & 0 \\ -\varphi^{j+1} & d_2^j \end{pmatrix}
\]
which gives the equations
\[
-d_1^j h_{12}^j + h_{12}^{j+1}d_2^j=0,
\]
\[
id_1^{j+1} = -d_1^{j}h_{11}^{j+1} - h_{11}^{j+2}d_1^{j+1} - h_{12}^{j+1}\varphi^{j+1},
\]
and
\[
id_2^j = -\varphi^j h_{12}^j + d_2^{j-1}h_{22}^j + h_{22}^{j+1} d_2^j
\]
for maps $h_{12}^j:A_2^j \to A_1^j$, $h_{21}^j:A_1^j \to A_2^{j-2}$, 
$h_{11}^j:A_1^j \to A_1^{j-1}$, and $h_{22}^j:A_2^j \to A_2^{j-1}$. 
This shows that 
$\varphi$ is a homotopy equivalence with inverse the (chain!) map $-h_{12}^{\cdot}$.
\end{proof}

Still assuming $\varphi$ is a homotopy equivalence, consider the distinguished triangle 
\[
A_1^{\cdot} \stackrel{\varphi}{\rightarrow} A_2^{\cdot} \to \cone(\varphi)^{\cdot}.
\]
Lemma \ref{lemma1} gives that $\cone(\varphi) \simeq 0$ so equation \eqref{eq1} shows that
Theorem \ref{2} (and hence Theorem \ref{1}) follows from the next result.

\begin{prop}\label{3}
Let $A^{\cdot}$ be a null-homotopic complex in $K^b(\mathcal{A})$, then $\chi(A^{\cdot})=0$ 
in $\K_{\oplus}(\mathcal{A})$.
\end{prop}

\begin{proof}
We may assume that 
\[
A^{\cdot} = A^0 \stackrel{d^0}{\longrightarrow} A^1 \stackrel{d^1}{\longrightarrow} \cdots 
\stackrel{d^{2k}}{\longrightarrow} A^{2k+1}
\]
contains all of the non-zero terms of $A^{\cdot}$. It suffices to show that
\[
\bigoplus_{i=0}^k A^{2i} \cong \bigoplus_{i=0}^k A^{2i+1}
\]
which we shall do by explicitly writing down the matrices giving the isomorphism.

Since $A^{\cdot}$ is null-homotopic there exist maps 
$A^j \stackrel{h^j}{\rightarrow} A^{j-1}$ so that 
\begin{equation*}
\id_j = d^{j-1}h^j+h^{j+1}d^j.
\end{equation*}
Using these equations, we can deduce the relations
\begin{equation*}
h^jh^{j+1}\cdots h^{j+2l+1} = 
d^{j-2}h^{j-1}h^j\cdots h^{j+2l+1} + h^j\cdots h^{j+2l+1}h^{j+2l+2}d^{j+2l+1}.
\end{equation*}
For instance, we can compute
\begin{align*}
h^j h^{j+1} &= h^j \id_j h^{j+1} \\
&= h^jd^{j-1}h^jh^{j+1} + h^jh^{j+1}d^jh^{j+1} \\
&= h^j h^{j+1} - d^{j-2}h^{j-1}h^jh^{j+1} + h^jh^{j+1} - h^jh^{j+1}h^{j+2}d^{j+1}
\end{align*}
and
\begin{align*}
h^j h^{j+1} h^{j+2} h^{j+3} &= h^j h^{j+1}\id_{j+1} h^{j+2} h^{j+3} \\
&= h^j h^{j+1}d^{j}h^{j+1} h^{j+2} h^{j+3} + h^j h^{j+1}h^{j+2}d^{j+1} h^{j+2} h^{j+3} \\
&= h^j h^{j+1} h^{j+2} h^{j+3} - h^j d^{j-1}h^{j}h^{j+1} h^{j+2} h^{j+3} \\
& \ \ \ \ \ + h^j h^{j+1} h^{j+2} h^{j+3} - h^j h^{j+1}h^{j+2}h^{j+3} d^{j+2} h^{j+3} \\
&= d^{j-2}h^{j-1}h^{j}h^{j+1} h^{j+2} h^{j+3} + h^j h^{j+1}h^{j+2}h^{j+3} h^{j+4} d^{j+3}.
\end{align*}
and similar computations (or induction on $l$) show the result in general.

Consider now the maps 
\[
R:\bigoplus_{i=0}^k A^{2i} \to \bigoplus_{i=0}^k A^{2i+1}
\]
and 
\[
L:\bigoplus_{i=0}^k A^{2i+1} \to \bigoplus_{i=0}^k A^{2i}
\]
given by 
\[
R=\begin{pmatrix}
d^0 &\alpha_0 h^2 &\alpha_1 h^2h^3h^4 &\alpha_2 h^2\cdots h^6 &\cdots &\alpha_{k-1}h^2\cdots h^{2k} \\
0 & d^2 &\alpha_0 h^4 &\alpha_1h^4h^5h^6 &\cdots & \alpha_{k-2}h^4\cdots h^{2k} \\
0 &0 &d^4 & \alpha_0 h^6 & \cdots & \alpha_{k-3}h^6\cdots h^{2k} \\
0 &0 &0 &d^6 &\cdots & \alpha_{k-4}h^8\cdots h^{2k} \\
\vdots &\vdots &\vdots &\vdots &\ddots &\vdots \\
0 &0 &0 &0 &\cdots &d^{2k} 
\end{pmatrix}
\]
and
\[
L=\begin{pmatrix}
\alpha_0 h^1 & \alpha_1 h^1h^2h^3 & \alpha_2 h^1\cdots h^5 &\alpha_3 h^1 \cdots h^7 
& \cdots & \alpha_k h^1\cdots h^{2k+1}\\
d^1 &\alpha_0 h^3 & \alpha_1 h^3h^4h^5 &\alpha_2 h^3\cdots h^7 &\cdots & \alpha_{k-1} h^3\cdots h^{2k+1} \\
0 & d^3 &\alpha_0 h^5 &\alpha_1 h^5h^6h^7 &\cdots &\alpha_{k-2} h^5\cdots h^{2k+1} \\
0 &0 &d^5 &\alpha_0 h^7 &\cdots &\alpha_{k-3} h^7\cdots h^{2k+1} \\
\vdots &\vdots &\vdots &\vdots &\ddots &\vdots \\
0 &0 &0 &0 &\cdots &\alpha_0 h^{2k+1}
\end{pmatrix}
\]
where $\{\alpha_k\}$ are integers defined by the recursion $\alpha_0=1$, $\alpha_1 = -1$, and 
\[
\alpha_k = - \sum_{j=0}^{k-1}\alpha_j\alpha_{k-1-j}.
\]
It is easy to see that in fact $\alpha_k = (-1)^k c_k$ where $c_k$ is the $k^{th}$ Catalan number. 

We now compute the entries of the matrices $RL$ and $LR$.
For $i<j$ we have
\begin{align*}
(RL)_{ij} 
&= \alpha_{j-i}d^{2i-2}h^{2i-1}\cdots h^{2j-1} + \alpha_0 \alpha_{j-i-1}h^{2i}\cdots h^{2j-1} + \cdots \\
& \ \ \ \ \ + \alpha_{j-i-1}\alpha_0 h^{2i} \cdots h^{2j-1} + \alpha_{j-i}h^{2i} \cdots h^{2j} d^{2j-1} \\
&= \alpha_{j-i}(d^{2i-2}h^{2i-1}\cdots h^{2j-1} + h^{2i} \cdots h^{2j} d^{2j-1} - h^{2i}\cdots h^{2j-1}) \\
&= 0
\end{align*}
and $(RL)_{ij}=0$ for $i>j$. We also compute 
\[
(RL)_{jj} =\alpha_0(d^{2j-2}h^{2j-1} + h^{2j}d^{2j-1}) = id_{2j-1}
\]
which shows that $RL = id$. Similarly, for $i<j$ we have
\begin{align*}
(LR)_{ij} 
&= \alpha_{j-i} d^{2i-3}h^{2i-2}\cdots h^{2j-2} + \alpha_0 \alpha_{j-i-1} h^{2i-1} \cdots h^{2j-2} + \cdots \\
& \ \ \ \ \ + \alpha_{j-i-1} \alpha_0 h^{2i-1} \cdots h^{2j-2} + \alpha_{j-i} h^{2i-1} \cdots h^{2j-1} d^{2j-2} \\
&= \alpha_{j-i} (d^{2i-3}h^{2i-2}\cdots h^{2j-2} + h^{2i-1} \cdots h^{2j-1} d^{2j-2} + h^{2i-1} \cdots h^{2j-2}) \\
&= 0
\end{align*}
and $(LR)_{ij}=0$ for $i>j$. We also see that 
\[
(LR)_{jj} = \alpha_0 (d^{2j-3}h^{2j-2} + h^{2j-1}d^{2j-2}) = id_{2j-2}
\]
so $LR = id$.
\end{proof}

We can slightly extend Proposition \ref{3} to the category $K^+(\mathcal{A})$ of bounded below complexes in 
$\mathcal{A}$. If $A^{\cdot}$ is such a complex and is null-homotopic, the infinite stable limit as 
$k\to \infty$ of the matrices $R$ and $L$ gives an isomorphism 
\[
\coprod_{i=-\infty}^{\infty} A^{2i} \cong \coprod_{i=-\infty}^{\infty} A^{2i+1}.
\]
If the category $\mathcal{A}$ is such that we can define a notion of Euler characteristic, this shows that 
null-homotopic complexes have zero Euler characteristic. In particular, this fact is used in \cite{Rose2}.

\newpage
\bibliographystyle{amsplain}
\bibliography{bib}

\end{document}